%% file: dcopt23.tex
\documentclass[11pt]{article}
\usepackage{amsmath,amssymb,a4wide,amsthm,color,graphicx,enumerate,booktabs}

\usepackage[]{algorithm2e}

\include{def}
\include{Format}

\author{Andreas L\"{o}hne \thanks{Friedrich Schiller University Jena, Department of Mathematics, 07737 Jena, Germany, andreas.loehne@uni-jena.de} \and Andrea Wagner \thanks{Vienna University of Economics and Business, Institute for Statistics and Mathematics, andrea.wagner@wu.ac.at}}

\title{Solving DC programs with a polyhedral component utilizing a multiple objective linear programming solver}

\begin{document}
\maketitle

\begin{abstract} 
A class of non-convex optimization problems with DC objective function is studied, where DC stands for being representable as the difference $f=g-h$ of two convex functions $g$ and $h$. In particular, we deal with the special case where one of the two convex functions $g$ or $h$ is polyhedral. In case   $g$ is   polyhedral, we show that a solution of the DC program can be obtained from a solution of an associated polyhedral projection problem. In case   $h$ is  polyhedral, we prove that a solution of the DC program can be obtained by solving a polyhedral projection problem and finitely many convex programs. Since polyhedral projection is equivalent to multiple objective linear programming (MOLP), a MOLP solver (in the second case together with a convex programming solver) can be used to solve instances of DC programs with polyhedral component. Numerical examples are provided, among them an application to locational analysis. 
\medskip

\medskip


\end{abstract}

\section{Introduction}

A function  $f: \R^n \to \R\cup\cb{\infty}$ is said to be {\em polyhedral convex} if its epigraph $\epi f:=\cb{(x,r)\in \R^n\times\R \st r \geq f(x)}$ is a polyhedral convex set. We consider the following DC optimization problem
\begin{equation}\label{p}\tag{P}
	\min\, [g(x) - h(x) ] \text{ subject to } x \in \dom g,
\end{equation}
where we assume that $g: \R^n \to \R\cup\cb{\infty}$ and $h: \R^n \to \R\cup\cb{\infty}$ are convex and at least one of the functions $g$ and $h$ is polyhedral convex. DC optimization problems are investigated for instance in \cite{Ale49,Dur02,FB97,Har59,HU86,LV96,MLS92,MLS97}.

The methods presented in this article are based on computing the vertices of  $\epi g$ in case   $g$ is polyhedral and the vertices of $\epi h^*$ in case  $h$ is polyhedral, where $h^*$ is the conjugate of $h$. It is shown that a solver for multiple objective linear programs (MOLP) can be used for this task. The aim of this paper is to show that global solutions of the considered class of DC-programs can be obtained by utilizing a MOLP-solver. The main advantage of this procedure is that the algorithms are very easy to implement. Nevertheless, the algorithms can be applied to certain non-convex problems in location theory, discussed below. We will show that they are competitive to some algorithms from the recent literature. It should be noted that the presented methods can be combined with other global optimization techniques in order to avoid to compute all vertices of the epigraphs. But then, a modification of the MOLP-solver is required, which would make the implementation more difficult.

The computation of vertices of an epigraph is related to the problem to project a polyhedron into a subspace. A {\em polyhedral convex set} $P$ in $\R^n$ (or {\em polyhedron} for short) is defined as the solution set of a system of finitely many linear inequalities, that is,
\begin{equation} \label{eq_h}
     P = \cb{x \in \R^n \st  B x \geq c}
\end{equation}
for a matrix $B \in \R^{m \times n}$ and a vector $c \in \R^m$. The pair $(B,\; c)$ is said to be an {\em H-representation} of $P$.
The well-known Weyl-Minkowski theorem states that every polyhedron can be expressed as the {\em generalized convex hull} of finitely many points $v^i \in \R^n$; $i=1,\dots, r$; $r \geq 1$ and finitely many directions $d^j \in \R^q\smz$; $j=1,\dots,s$; $s \geq 0$, that is
\begin{equation} \label{eq_v}
	 P=\cb{x\in \R^n\st x=V\lambda + D\mu,\; e^T \lambda= 1,\; \lambda \geq 0,\; \mu \geq 0},
\end{equation}
where $v^i$ and $d^j$ are the columns of $V\in \R^{n\times r}$ and $D \in \R^{n\times s}$, respectively. Further we denote by $e:=(1,\dots,1)^T$ the all-one-vector.
We say that $(V,D)$ is a {\em V-representation} of the polyhedron $P$.
Both H-representation and V-representation are special cases of the {\em projection representation} or {\em P-representation} $(B,C,c)$, that is
\begin{equation}\label{eq_p}
	P = \cb{x\in \R^n \st \exists u \in \R^k: B x + C u \geq c},
\end{equation}
for matrices $B \in \R^{m\times n}$, $C\in \R^{m\times k}$, $c\in \R^{m\times 1}$. In Sections \ref{sec_rep} and \ref{sec_app} we will see that in many situations only a P-representation of a polyhedron is known. An H-representation can be obtained from a P-representation by Fourier-Motzkin elimination and a V-representation can be obtained from an H-representation by vertex enumeration, for instance, by using the double description method by Motzkin et al. \cite{Motzkin53}. However, the first part of this procedure (Fourier-Motzkin elimination) is not tractable if $k$ is large. It has been shown in \cite{LoeWei15} that both an H-representation and a V-representation can be obtained from a P-representation by solving a multiple objective linear program (MOLP). This follows from the equivalence between a polyhedral projection problem (which is roughly speaking the problem to obtain a V-representation from a P-representation \eqref{eq_p}) and multiple objective linear programming, as shown in \cite{LoeWei15}. By {\em equivalence} the authors of \cite{LoeWei15} understand that a solution of the one problem can be ``easily'' (for details see \cite{LoeWei15} or Section \ref{sec_pp} below) obtained from a solution of the other problem. In order to compute a V-representation of \eqref{eq_p}, a multiple objective linear program with $n+1$ objective functions needs to be solved.

Multiple objective linear programs can be solved by Benson's algorithms \cite{Benson98}. Numerical examples for up to 10 objective functions are provided in \cite{Csirmaz13} and \cite{LoeWei16}. A more direct way to solve the polyhedral projection problem is the {\em convex hull method} by Lassez and Lassez \cite{LasLas90}. This method, however, is closely related to the dual variant of Benson's algorithm for multiple objective linear programs \cite{EhrLoeSha12}, which has been developed independently.

We close this section with some notation. We denote by $v_i$ the components of a vector $v$ and by $M^i$ the rows of a matrix $M$. As mentioned above, we set $e=(1,\dots,1)^T$. Given a convex function $g: \R^n \to \R\cup\cb{\infty}$ we denote by $\dom g:= \cb{x\in \R^n\st g(x) \neq \infty}$ the domain of $g$, and by $g^*:\R^n\to\R\cup\cb{\infty}$, $g^*(x^*):=\sup_{x\in \dom g} [{x^*}^Tx - g(x)]$ the conjugate of $g$.  {If $\dom g\neq \emptyset$ then the function $g$ is called proper.}
The recession cone of a polyhedron \eqref{eq_h} is the polyhedral convex cone  $0^+P := \cb{x \in \R^n \st  B x \geq 0}$. The closure, interior, boundary of a set $A$ is denoted, respectively, by $\cl A$, $\Int A$, $\bd A$.

\section{Representing polyhedral convex functions}\label{sec_rep}

Polyhedral convex functions can be represented by their epigraphs. The following definition is based on a P-representation  of the epigraph. 

\begin{definition} A matrix $A \in \R^{m \times (n+1+k+1)}$ is called a {\em representation} of a polyhedral convex function $f: \R^n \to \R\cup\cb{\infty}$ if
\begin{equation}\label{eq1}
	\epi f = \cb{ \begin{pmatrix}x\\r\end{pmatrix} \in \R^{n+1} \bigg|\; \exists u \in \R^k: A \begin{pmatrix}x\\r\\u\\-1\end{pmatrix}  \geq 0}.	
\end{equation}
\end{definition}

If the matrix $A$ is partitioned as $A=(B,b,C,c)$, where $B \in \R^{m\times n}$, $b \in \R^m$, $C \in \R^{m \times k}$ and $c \in \R^m$, we obtain
	\begin{equation}\label{eq5}
	  \epi f = \cb{ \begin{pmatrix}x\\r\end{pmatrix} \in \R^{n+1} \bigg|\; \exists u \in \R^k: B x + b r + C u \geq c}.
	\end{equation}

The following examples motivate the definition of a representation. We show that for many (classes of) polyhedral convex functions a representation can be obtained without any essential computational effort.

\begin{example}\label{ex0a}
	Consider finitely many affine functions $f_i(x)={D^i} x + d_i$ $(i=1\dots,m)$, a matrix 
	$P \in \R^{k\times n}$ and a vector $p \in \R^k$. Then 
	$$ f(x):= \left\{
	              \begin{array}{cl}
	                  \displaystyle\max_{i\in \cb{1,\dots,m} }f_i(x) & \text{ if } P x \geq p\\
		         + \infty & \text{ otherwise}
		    \end{array}\right.    $$
	is a polyhedral convex function with representation
	$$ A = \begin{pmatrix} -D & e & 0 & d \\ \phantom{-}P & 0 & 0 & p\end{pmatrix},$$
	where $D^1,\ldots,D^m$ are the rows of $D\in\R^{m\times n}$  and $d:=(d_1,\ldots,d_m)\in\R^m$. 
\end{example}

\begin{example}\label{ex0b}
	Let $f$ be a polyhedral convex function and let $(V,D)$ be a V-representation of $\epi f$, that is, $V \in \R^{(n+1)\times r}$ {with} $r \geq 1$ {and} $D \in \R^{(n+1)\times s}$ {where} $s \geq 0$,
	$$\epi f := \cb{z \in \R^{n+1} \st \exists \lambda \in \R^r, \mu \in \R^s:\; z = V \lambda + D \mu,\; e^T \lambda = 1,\; \lambda \geq 0,\; \mu \geq 0}.$$
	 Then 
	 $$ A = \begin{pmatrix}(B,b) & C & c \end{pmatrix} = \begin{pmatrix}
	 \mp I & (\pm V, \pm D) & 0 \\
		 0 & I & 0\\
	 0 & (\pm e^T, 0) & \pm 1
	 \end{pmatrix} $$
	 is a representation of $f$, where $\pm$ and $\mp$ are used to express the corresponding equations by inequalities.
\end{example}

\begin{example}\label{ex0c}
	Let $g,h: \R^n \to \R \cup \cb{\infty}$ be polyhedral convex functions with representations
	$$ A_g = \begin{pmatrix} B_g & b_g&C_g&c_g\end{pmatrix} \quad \text{and} \quad A_h = \begin{pmatrix} B_h & b_h&C_h&c_h\end{pmatrix} .$$
The {\em infimal convolution} of $g$ and $h$ is the polyhedral convex function $f$ with 
$$ \epi f = \epi g + \epi h.$$	
	 Thus, its representation is $A_f = \begin{pmatrix} B_f & b_f&C_f&c_f\end{pmatrix}$ with
	 $$ (B_f\;b_f) = \begin{pmatrix} \pm I \\ 0 \\ 0 \end{pmatrix} \quad C_f= \begin{pmatrix}  \mp I & 0 & \mp I & 0 \\
	 					 (B_g\; b_g) & C_g & 0 & 0  \\
						0 &  0 & (B_h\; b_h) & C_h  \end{pmatrix}  \quad c_f = \begin{pmatrix}
							0 \\ c_g \\ c_h
						\end{pmatrix}.
						$$
\end{example}

\begin{example}\label{ex0d}
	Let a representation $A = \begin{pmatrix} B&b&C&c\end{pmatrix}$ of a polyhedral convex function $f :\R^n \to \R\cup\cb{\infty}$ be given, then
	$$  A^* = \begin{pmatrix}
		\pm I & 0 & \pm B^T & 0 \\
		0 & 0 & \pm C^T & 0 \\
		0 & 0 & \pm b^T & \pm 1 \\
		0 &  1&  c^T & 0 \\
		0 & 0 & I & 0 
	\end{pmatrix}$$
	is a representation of the conjugate $f^* :\R^n \to \R\cup\cb{\infty}$ of $f$. The details are shown in the following proposition.
\end{example}

	\begin{proposition}\label{rep_conj} For a polyhedral convex function $f :\R^n \to \R\cup\cb{\infty}$, the following two statements are equivalent:
		\begin{equation}\label{eq7}
			\epi f = \cb{ \begin{pmatrix}x\\r\end{pmatrix} \in \R^{n+1} \bigg|\; \exists u \in \R^k: B x + b r + C u \geq c};
			\end{equation}
		\begin{equation}\label{eq8}
			\epi f^* = \cb{ \begin{pmatrix}x^*\\r^*\end{pmatrix} \in \R^{n+1} \bigg|\; \exists v \in \R^m_+: B^T v + x^* = 0,\; b^T v = 1, \;  C^T v = 0, \; c^T v + r^* \geq 0}.
		\end{equation}
	\end{proposition}
	\begin{proof} Assume that \eqref{eq7} is satisfied. Let $(x^*,r^*) \in \epi f^*$, i.e., 
		$$ r^* \geq f^*(x^*) = \sup_{x \in \R^n} \of{{x^*}^T x - f(x)} = \sup_{(x,r)\in \epi f}  \of{{x^*}^T x - r}.$$
	This is equivalent to the implication
	$$ (x,r)\in \epi f \implies r^* \geq {x^*}^T x - r.$$
	By \eqref{eq7}, this can be expressed as 
	$$ B x + b r + C u \geq c \implies r^* \geq {x^*}^T x - r.$$
	By Farkas' lemma (see e.g. \cite[Theorem 7.20]{Gueler10}) we obtain the existence of $v \in \R^m_+$ such that
	$$ B^T v + x^* = 0,\; b^T v = 1, \;  C^T v = 0, \; c^T v + r^* \geq 0, $$
	i.e., \eqref{eq8} holds. The opposite implication follows likewise by taking into account that $f=f^{**}$ for a polyhedral convex function $f$.	
	\end{proof}
	
The list of examples could be extended. Note that, in general, it is much more expensive to compute an H-representation or a V-representation of the epigraph of a polyhedral convex function. This problem, however, is important for our studies and will be discussed in Section \ref{sec_pp}.  

\section{Maximizing a convex function over a polyhedron}

A well-known result on maximizing a convex function over a polytope $P$ is the attainment of the maximum {in at least one vertex} of $P$. We discuss in this short section the case of unbounded polyhedra.

Let $f:\R^n \to \R \cup \cb{\infty}$ be convex and let $P$ be a polyhedron. Assume that $P$ has a vertex.
We consider the problem
\begin{equation}\label{c}
\tag{C} \max f(x) \quad\text{ subject to } \quad x \in P.
\end{equation}
\begin{proposition}\label{prop1}
If Problem \eqref{c} has an optimal solution, then some vertex $x^*$ of $P$ is an optimal solution of \eqref{c}.
\end{proposition}
\begin{proof} We denote by $v^1,\dots,v^k$ the vertices of $P$ and set $B:= \conv\cb{v^1,\dots,v^k}$. Convexity of $f$ implies that there is a vertex $v^j$ of $B$ which is an optimal solution of 
	$$ \max f(x) \quad\text{ subject to }\quad x \in B.$$
Assume there is $x \in P \setminus B$ with $f(x) > f(v^j)$. The point $x$ can be represented as the sum $x=b+d$ of a point $b \in B$ and a direction $d \in (0^+ P) \smz$. Then, $f(b+d) > f(v^j) \geq f(b)$. Convexity of $f$ implies that $f(b+nd) \to \infty$ as $n \to \infty$, which contradicts the existence of an optimal solution.
\end{proof}

Then next statement characterizes the existence of a solution.

\begin{proposition}\label{prop_d1}
	Problem \eqref{c} is unbounded if and only if there exists a vertex $v$ of $P$, an extreme direction $d$ of $P$ and some $\beta > 0$ with
	$$ f(v)<f(v+ \beta d).$$
\end{proposition} 
\begin{proof} Let \eqref{c} be unbounded. There is a sequence $(x_\nu)$ in $P$ with $(f(x_{\nu})) \to \infty$ as $\nu \to \infty$. For all $\nu \in \N$, we have 
	$$ x_{\nu} = \sum_{i=1}^r \lambda^\nu_i v^i + \sum_{j=1}^s \mu^\nu_j d^j,\qquad \sum_{i=1}^r \lambda^\nu_i = 1, \qquad \lambda^\nu_i \geq 0,\quad \mu^\nu_j \geq 0,$$
	where $v^i$ are the vertices and $d^i$ the extreme directions of $P$. Setting 
	$$ \alpha^\nu_{ij} := \frac{\lambda^\nu_i \mu^\nu_j}{\beta^\nu} \quad \text{ and } \quad \beta^\nu := \sum_{k=1}^s \mu^\nu_k,$$	
	we obtain 
	$$ x_{\nu} = \sum_{i=1}^r\sum_{j=1}^s \alpha^\nu_{ij}(v^i + \beta^\nu d^j), \qquad \sum_{i=1}^r\sum_{j=1}^s \alpha^\nu_{ij} = 1,\qquad \alpha^\nu_{ij}\geq 0.$$
	By Proposition \ref{prop1}, we have $\beta^\nu \neq 0$ for $\nu$ being sufficiently large. Convexity implies that 
	$$ \sum_{i=1}^r\sum_{j=1}^s \alpha^\nu_{ij} f(v^i + \beta^\nu d^j) $$
	tends to infinity. Hence $f(v^i + \beta^\nu d^j) \to \infty$ for at least one $i$ and $j$, which proves the first part of the statement.
	
	The converse implication follows from convexity of $f$ along the ray $\cb{v+\beta d \st \beta \geq 0}$, which belongs to $P$.
\end{proof}

\section{The case of $g$ being polyhedral} \label{sec_g}

The algorithm considered in this section is based on an enumeration of the vertices of $\epi g$.  We start with a reformulation of the DC program \eqref{p}.

\begin{proposition}
Problem \eqref{p} can be expressed equivalently as
\begin{equation}\label{p_1}
	\min  [r - h(x)] \quad\text{ subject to }\quad (x,r) \in \bd \epi g.
\end{equation}
\end{proposition}
\begin{proof} Since $g:\R^n \to \R\cup \cb{\infty}$ (i.e. the value $-\infty$ is not allowed), for the graph $\gr g:=\cb{(x,g(x)) \st x \in \dom g}$ of $g$, we have
	$$ \gr g \subseteq \bd\epi g \subseteq \epi g = \gr g + (\cb{0}\times \R_+). $$
	Let $\bar x$ be an optimal solution of \eqref{p}. For all $(x,r) \in \epi g$ (and in particular for all $(x,r) \in \bd\epi g$) we have $g(\bar x) - h(\bar x) \leq g(x)-h(x) \leq r - h(x)$.
	Since $(\bar x,g(\bar x)) \in\bd\epi g$, we conclude that $(\bar x,g(\bar x))$ solves \eqref{p_1}.
	
	Vice versa, let $(\bar x,\bar r) \in  \bd\epi g$ be a solution of \eqref{p_1}, then	
	$$ \forall (x,r) \in \gr g \subseteq \bd\epi g: \quad \bar r - h(\bar x) \leq r - h(x).$$
We get $\bar r=g(\bar x)$ and see that $\bar x \in \dom g$ solves  \eqref{p}.
\end{proof}

Assume now that $g:\R^n \to \R\cup \cb{\infty}$ is polyhedral convex. Let $\cb{F_i\st i=1,\dots,k}$ be the (finite) set of all facets of $\epi g$. Then, it is evident that
$$ \bd \epi g = \bigcup_{i=1,\dots,k} F_i.$$
The following statement is now obvious. 
\begin{corollary} For $i=1,\dots,k$, let $(x^i,r^i)$ be an optimal solution of
\begin{equation} \label{pi}\tag{P$_i$}
	\min  [r - h(x)] \quad\text{ subject to }\quad (x,r) \in F_i.
\end{equation}
Let $j \in \argmin\cb{r^i-h(x^i)\st i=1,\dots,k}$. Then, $(x^j,r^j)$ is an optimal solution of \eqref{p}.
Problem (P) has an optimal solution if and only if, for all $i \in \cb{1,\dots,k}$, \eqref{pi} has an optimal solution.
\end{corollary}

Assume that \eqref{p} has an optimal solution. Then, for every  $i \in \cb{1,\dots,k}$, Problem \eqref{pi} has an optimal solution which is, by Proposition \ref{prop1}, a vertex of the facet $F_i$ of $\epi g$. Since a vertex of a facet $F$ of $\epi g$ is also a vertex of $\epi g$, an optimal solution of \eqref{p} can  always be found among the vertices of $\epi g$, denoted $\vertex \epi g$.

\begin{corollary} Assume that \eqref{p} has an optimal solution. Every optimal solution $(\bar x,\bar r)$ of 
\begin{equation}\label{pe}	
	\min [r - h(x)] \quad\text{ subject to }\quad (x,r) \in \vertex \epi g.
\end{equation}
is also an optimal solution of \eqref{p}.
\end{corollary}

If a representation $A$ of the polyhedral convex function $g$ is given, the vertices of $\epi g$ can be computed by solving a polyhedral projection problem, see Section \ref{sec_pp}. If $\vertex \epi g$ has been computed, it remains to determine the minimum over finitely many points in problem \eqref{pe}. We obtain the following solution procedure for the DC optimization problem \eqref{p}.

\bigskip\noindent
\begin{assumptions}\label{ass_primal}
	\mbox{}
\begin{enumerate}[(i)]
	\item (P) has an optimal solution,	\label{a1}
	\item $g$ is polyhedral convex,\label{a2}
	\item $\epi g$ has a vertex,\label{a3}
	\item a representation $A$ of $g$ is known.   \label{a4}
\end{enumerate}
\end{assumptions}

\bigskip\noindent
\begin{algorithm+}\label{primal_alg}
{Assume that Assumptions \ref{ass_primal}~\eqref{a1}--\eqref{a4} are satisfied.}
	
\begin{algorithm}[H]
	\KwData{representation $A$ of $g$,  function $h$.}
	\KwResult{$x^j$ is an optimal solution of \eqref{p}. }
	\begin{enumerate}[(1)]
		\item Compute the vertices $(x^1,r^1),\dots,(x^k,r^k)$ of $\epi g$ by solving a polyhedral projection problem, see Section \ref{sec_pp}.
		\item Choose $(x^j,r^j)$ such that $r^j - h(x^j) = \displaystyle\min_{i=1,\dots,k} [r^i - h(x^i)]$.
	\end{enumerate}
\end{algorithm}
\end{algorithm+}

\begin{corollary}
	Algorithm \ref{primal_alg} works correctly if  Assumptions {\ref{ass_primal}~}\eqref{a1}--\eqref{a4} are satisfied.
\end{corollary}
\begin{proof}
	This follows from the considerations above.
\end{proof}

\begin{remark}
	{
	Of course one can add linear constraints to Problem \eqref{p} in case $g$ is polyhedral. These constraints can be included into a representation of $g$, see e.g. Example \ref{ex:loc_an}.}
\end{remark}

\section{The case of $h$ being polyhedral}\label{sec_h}

We assume now that the function $h$ in Problem \eqref{p} is polyhedral. We consider the Toland-Singer dual problem  of \eqref{p}, see \cite{Sin79,Tol78}, that is,
\begin{equation}\label{d}\tag{D}
	\min_{y \in \dom h^*} [h^*(y) - g^*(y)],
\end{equation}
where $h^*(y) := \sup_{x \in \R^n} [{y}^Tx - h(x)]$ is the conjugate of $h$ and likewise for $g$. Due to the duality theory by Toland and Singer the optimal objective values of \eqref{p} and \eqref{d} coincide  in case {of $h$ being closed}. This condition is satisfied as we assume $h$ to be a polyhedral convex function.

Since $g^*$ is convex and $h^*$ is polyhedral convex, Problem \eqref{d} can be solved by applying Algorithm \ref{primal_alg} in Section \ref{sec_g}. As pointed out below, a solution of \eqref{p} can be computed from a solution of \eqref{d}. We suppose the following assumptions, all being related to  the given problem \eqref{p}:

\bigskip\noindent
\begin{assumptions}\label{ass_dual}
	\mbox{}
%\textbf{Assumptions:}
\begin{enumerate}[(i)]
	\item (P) has an optimal solution,	\label{b1}
	\item $h$ is polyhedral convex, \label{b2}
	\item $\dim \epi h = n+1$,\label{b3}
	\item a representation $A$ of $h$ is known,\label{b4}
	\item $g$ is closed. \label{b5}
	\item  A solution of the following problem exists whenever the problem is bounded for the parameter $y \in \R^n$:
    \begin{equation}\label{p3}
    	\min_{x \in \R^n} [g(x) - {y}^T x].
    \end{equation}	\label{b6}
\end{enumerate}
\end{assumptions}
The following result relates solutions of \eqref{p} and solutions of \eqref{d} to each other. We denote by $\partial f(\bar x):= \cb{y \in \R^n\st \forall x \in \R^n: f(x) \geq f(\bar x) + {y}^T(x-\bar x)}$ the {\em subdifferential} of a convex function $f:\R^n \to \R \cup \cb{ \infty}$ at $\bar x \in \dom f$. We set $\partial f(\bar x) := \emptyset$ for $\bar x \not\in\dom f$. For convenience of the reader and {with regard to} Remark \ref{rem1} below, we provide a short {proof} of the following statement.
\begin{proposition}{(e.g. \cite[Proposition 4.7]{HT99} or \cite[Proposition 3.20]{Tuy98})} \label{p_DC}
	Let $g:\R^n \to \R\cup\cb{\infty}$ and $h:\R^n \to \R\cup\cb{\infty}$ be proper convex functions. Then the following holds true.
	\begin{enumerate}[(i)]
	\item If $x$ is an optimal solution of \eqref{p}, then each $y \in \partial h(x)$ is an optimal solution of \eqref{d}.			
	\end{enumerate}	
	If, in addition, $g$ and $h$ are closed, a dual statement holds:
	\begin{enumerate}[(i)]
	\item[(ii)] If $y$ is an optimal solution of \eqref{d}, then each $x \in \partial g^*(y)$ is an optimal solution of \eqref{p}.
	\end{enumerate}
\end{proposition}
\begin{proof} {Let $y\in \dom h$ be optimal for \eqref{d} and $x\in\partial g^*(y)$.} To prove statement (ii), we note that we have
	\begin{equation}\label{eq:sol_d}
		\forall y^* \in \dom h^*: \; h^*(y^*)-g^*(y^*) \geq h^*(y)-g^*(y),
	\end{equation}
	\begin{equation}\label{eq:f_e}
		g(x) + g^*(y) = {y}^Tx,
	\end{equation}
	where the latter equation requires $g$ being closed, compare \cite[Theorem 23.5]{Rockafellar72}.
	Moreover we use Fenchel's inequality, that is, $g^*(y^*) + g({z}) \geq  {y^*}^T{z}$ for all ${z}\in \R^n$ and all $y^* \in \R^n$ (and likewise for $h$). For all $y^* \in \dom h^*$ and all $x \in \dom g$ we get
\begin{align*}
	g({z}) + h^*(y^*)- {y^*}^T{z} & \geq h^*(y^*)-g^*(y^*) \stackrel{\eqref{eq:sol_d}}{\geq} h^*(y)- g^*(y)\\
								& \geq  {y}^Tx - h(x) - g^*(y) \stackrel{\eqref{eq:f_e}}{=} g(x) - h(x).
\end{align*}
Taking the infimum over $y^* \in \dom h^*$, we get
$$ \forall {z} \in \dom g:\; g({z}) - h^{**}({z}) \geq g(x) - h(x).$$
The claim follows as $h^{**}({z}) = h({z})$. Here we use that $h$ is assumed to be closed, compare \cite[Theorem 12.2]{Rockafellar72}.
The proof of (i) is analogous but neither $g$ nor $h$ needs to be closed, compare \cite[Theorem 23.5]{Rockafellar72} and take into account the fact that no biconjugation argument is necessary.
\end{proof}

\begin{remark} \label{rem1} In \cite[Proposition 4.7]{HT99} the assumption of $g$ and $h$ being closed for statement (ii) of Proposition \ref{p_DC} is missing. Moreover in \cite[Proposition 3.20]{Tuy98} closedness of $g$ needs to be assumed. Indeed, consider the following example.
	Let 
	$$ f:\R\to \R\cup\cb{\infty},\quad f(x):=\left\{\begin{array}{rcl} \infty & \text{ if } & x < 0\\
	1 & \text{ if } & x=0\\
	0 & \text{ if } & x>0.		
	\end{array}\right. $$
	Then $f$ is not closed and we have
	$$ f^*:\R\to \R\cup\cb{\infty},\quad f^*(y):=\left\{\begin{array}{rcl} 0 & \text{ if } & y \leq 0\\
	\infty & \text{ if } & y>0.		
	\end{array}\right. $$
	Set $g=f$ and $h=\cl f$. Then $y=0$ solves \eqref{d} and $0 \in \partial g^*(0)$. But $x=0$ is not a solution of \eqref{p}. This shows the  failure when $g$ is not closed. 
	
	Set $g=\cl f$ and $h=f$. Then $y=0$ solves \eqref{d} and $1 \in \partial g^*(0)$. But $x=1$ is not a solution of \eqref{p}. This shows the failure when $h$ is not closed. 	
\end{remark}

\begin{proposition}{(e.g. \cite[Theorem 23.5]{Rockafellar72})}\label{p5} Let $g:\R^n \to \R$ be a proper closed convex function. Then $\bar x \in \partial g^*(y)$ if and only if $\bar x$ is an optimal solution of \eqref{p3}.
\end{proposition}

The next result provides a consequence of Assumption {\ref{ass_dual}~}\eqref{b3}.

\begin{proposition} \label{p4} Let $h:\R^n\to \R\cup\cb{\infty}$ be a polyhedral convex function such that $\dim \epi h = n+1$. Then $\epi h^*$ has a vertex.
\end{proposition}
\begin{proof} 
	The assumption implies that there is a point $x \in \Int\dom h$, such that $h$ is affine in a neighborhood of $x$. Whence $\partial h(x)$ is a singleton set $\cb{\bar{y}}=\cb{\nabla h(x)}$. The Fenchel-Young inequality implies
	$$ \forall z \in \R^n :\quad h^*(z) \geq  {z}^Tx - h(x).$$
		Using \cite[Theorem 23.5]{Rockafellar72}, we see that equality holds if and only if $z=\bar{y}$. This means that 
	$$ H=\cb{(z,r^*) \in \R^{n+1} \st  r^* =  {z}^Tx - h(x)}$$
	is a supporting hyperplane of $\epi h^*$ such that $\epi h^* \cap H = \cb{(\bar{y},h^*(\bar{y}))}$, i.e.,
	$(\bar{y},h^*(\bar{y}))$ is a vertex of $\epi h^*$.
\end{proof}

We obtain the following algorithm:

\begin{algorithm+}\label{dual_alg}
	{Assume that Assumptions \ref{ass_dual}~\eqref{b1}--\eqref{b6} are satisfied.}
	
	\begin{algorithm}[H]
		\KwData{representation $A$ of $h$, {function} $g$.}
		\KwResult{$\bar x$ is an optimal solution of \eqref{p}. }
		\begin{enumerate}[(1)]
			\item Compute the vertices $(y^1,s^1),\dots,(y^k,s^k)$ of $\epi h^*$ by solving a polyhedral projection problem, see Section \ref{sec_pp}.
			\item Determine $(y^*,s^*)$ such that $s^* - g^*(y^*) = \min_{i=1,\dots,k} [s^i - g^*(y^i)]$,
			\item Compute an optimal solution $\bar x$ of \eqref{p3} with parameter $y^*$.
		\end{enumerate}
	\end{algorithm}
\end{algorithm+}

\smallskip

\begin{theorem}
	Let the Assumptions {\ref{ass_dual}~}\eqref{b1}--\eqref{b6} be satisfied. Then Algorithm \ref{dual_alg} works correctly.
\end{theorem}
\begin{proof}
Note first that \eqref{p} has a solution $\bar x$ by Assumption {\ref{ass_dual}~}\eqref{b1}. A solution $\bar x$ of (P) is by definition a point in $\dom g \cap \dom h$. Since $h$ is polyhedral convex and $\bar x \in \dom h$, we have $\partial h(\bar x)\neq \emptyset$, see e.g. \cite[Theorem 23.10]{Rockafellar72}. By Proposition \ref{p_DC} (i), a point $\bar{y} \in \partial h(\bar x)$ is a solution to \eqref{d}. Whence a solution of \eqref{d} exists. In Steps (1) and (2) of the algorithm, a solution $\bar{y}$ of \eqref{d} is computed. To see this, we need to check  Assumptions {\ref{ass_dual}~}\eqref{b1}--\eqref{b6} of Algorithm \ref{primal_alg} in Section \ref{sec_g}, reformulated for problem \eqref{d}:
	\begin{itemize}
		\item (D) has an optimal solution, as shown above.
		\item $h^*$ is polyhedral convex as so is $h$.
		\item $\epi h^*$ has a vertex. This follows from Proposition \ref{p4}. 
		\item A representation $A$ of $h^*$ is known, as it can be obtained from a representation of $h$, see Example \ref{ex0d} in Section \ref{sec_rep}.
	\end{itemize}
	Since $\bar{y}$ computed in Steps (1) and (2) is a solution to (D), we have $\bar{y} \in \dom g^*$. Hence Problem \eqref{p3} is bounded (by the definition of conjugates). In this case, by Assumption {\ref{ass_dual}~}\eqref{b4}, an optimal solution $\bar x$ of \eqref{p3} exists. By Proposition \ref{p5}, we get $\bar x \in \partial g^*(\bar{y})$. Proposition \ref{p_DC} (ii) implies that $\bar x$ is an optimal solution of \eqref{p}.
\end{proof}

\begin{remark}\label{rev_rem1}
	
	One can add convex constraints $f_i(x)\leq 0$ $(i=1,\dots,m)$ to Problem \eqref{p} in case $h$ is polyhedral. Setting $g(x) = \infty$ whenever one constraint is violated, we maintain convexity of $g$. If $g$ and $f_i$ $(i=1,\dots,m)$ are closed, so is the modified $g$. The only difference in the algorithm is that the constraints have to be added to Problem \eqref{p3}.
\end{remark}

\begin{remark}\label{rev_rem2}
	 An alternative approach for minimizing a DC function $f=g-h$ where $h$ is polyhedral is the following: The domain of $h$ can be ``partitioned'' into finitely many convex polyhedral sections $S_i$ (such that the intersection of the interior of any $S_i$ and $S_j$ for $i\neq j$ is empty) such that $h$ is affine on each $S_i$.  Minimizing $g-h$ over such a section $S_i$ is a convex program with linear constraints. Solving such a convex program for each $S_i$ and choosing the one with smallest optimal value, we obtain an optimal solution of the original problem. To compute the sections $S_i$ one needs to solve a polyhedral projection problem (or similar) involving the epigraph of $h$, which is not easier than Step 1 in Algorithm \ref{dual_alg}. The advantage of our algorithm is as follows: First, in case the conjugate of $g$ is known, we have to solve only one convex program. Secondly, only unconstrained convex programs need to be solved in our algorithm. 
	
\end{remark}

\section{Polyhedral projection via multiple objective linear programming} \label{sec_pp}

\noindent
For matrices $B \in \R^{m \times n}$, $C\in \R^{m \times k}$ and a vector $c \in \R^m$, we consider the problem 
\begin{equation}\label{PP}\tag{PP}
\text{compute }\quad	 Y = \cb{x \in \R^n \st \exists u \in \R^k: B x + C u \geq c}.
\end{equation}
Solving \eqref{PP} {essentially means finding } a V-representation of $Y$, see \cite{LoeWei15} for details.
Using the terminology of Section \ref{sec_rep}, \eqref{PP} provides a P-representation of $Y$.
The polyhedron $Y$ can be seen as the projection of the polyhedron 
\begin{equation}\label{X}
	  X = \cb{(x,u) \in \R^n \times \R^k \st B x + C u \geq c}
\end{equation}
{onto } $\R^n$. 

In \cite{LoeWei15}, the equivalence between the projection problem \eqref{PP} and multiple objective linear programming (MOLP) is shown. In particular, \cite[Theorem 3]{LoeWei15} yields that a V-representation of $Y$ can be obtained by solving the multiple objective linear program
\begin{equation*}\label{molp}
	\tag{MOLP}	\min_{x\in\R^n,u\in\R^k} \begin{pmatrix} x \\ -e^T x\end{pmatrix} \quad \text{ subject to }\quad B x + C u \geq c,
\end{equation*}
where $e=(1,\dots,1)^T$. 

A MOLP-solver such as {\sl Bensolve} \cite{bensolve,LoeWei16} computes not only a solution to \eqref{molp} in the sense of \cite{HeyLoe11, Loehne11} but also a V-representation of the so-called {\em upper image}. The upper image of \eqref{molp} is the set
$$ \PP := \cb{z \in \R^{n+1} \bigg|\; \exists u \in \R^k,\, \exists x \in \R^n:\; z \geq \begin{pmatrix} x \\ -e^T x\end{pmatrix}, \; Bx + Cu \geq c},$$
whereas the {\em image} of \eqref{molp}, in the literature also called {\em feasible set in the objective space}, is
$$ Q :=  \cb{z \in \R^{n+1} \bigg|\; \exists u \in \R^k,\, \exists x \in \R^n:\; z = \begin{pmatrix} x \\ -e^T x\end{pmatrix}, \; Bx + Cu \geq c}.$$
Clearly, we have
$$ \PP = Q + \R^{n+1}_+$$
and
$$ Q \subseteq \cb{z \in \R^{n+1} |\; e^T z = 0}.$$
Thus, a V-representation of $Q$ can be obtained from a V-representation of $\PP$ by omitting all points and directions $x$ with $e^T z \neq 0$. 
Evidently, a V-representation of $Y$ can be obtained from a V-representation of $Q$ by omitting the last component of each vector.

\section{Application and numerical results} \label{sec_app}

 First, we solve an instance of a location problem, then we compare our results with two algorithms from the literature and finally we present two new examples. We used the VLP-Solver Bensolve version 2.0.1 \cite{bensolve, LoeWei16} to compute the polyhedral projection problems. VLP stands for {\em vector linear programming} which is a generalization of MOLP. The remaining part was implemented in Matlab\textregistered\ R2015b. In particular, we use the \verb+linprog+ command if problem \eqref{p3} needs to be solved, which reduces to an LP here. All computations were run on a computer with Intel\textregistered\ Core\texttrademark\ i5 CPU with 3.3 GHz. 

As both $g$ and $h$ are polyhedral convex functions in our first two examples, we obtain two alternatives to solve the problems. Algorithm \ref{primal_alg} developed in Section \ref{sec_g} (case $g$ polyhedral) is referred to as the {\em primal algorithm}, whereas Algorithm \ref{dual_alg} of Section \ref{sec_h} (case $h$ polyhedral) is called the {\em dual algorithm}. As we will see, both algorithms are justified.

\begin{example} \label{ex:loc_an}

The classical location theory is dedicated to the problem of locating a new facility such that all relevant distances, e.g.\ to  customers, are minimized \cite{Dre95,DKSW02,DM85,EL95,Ham95,Nic05,Web09}. Nevertheless, since some facilities also cause negative effects like noise, stench or even danger, such facilities need to be established as far away as possible from nature reserves and residential areas. As examples one may think of an airport, railway station, industrial plant, dump site, power plant or wireless station. 

The problem of locating such a so-called semi-obnoxious facility consists of minimizing the weighted sum of distances to the so-called attraction points (due to economical reasons) and to maximize the weighted sum of distances to the so-called repulsive ones (in order to avoid negative effects).

First attempts to solve location problems with such undesirable facilities appeared in the 1970's \cite{CG78,DasWhi80,GD75}. Algorithms based on a DC formulation are presented for instance in \cite{CHJT92,DW91,MF94,TAKZ95}.

It it reasonable to establish the new facility in a given area, city or state, which we assume to be a bounded polyhedral set $\mathcal{P}=\set{x\in \R^n|\,Px\geq p}\subseteq\dom h$.   The DC location problem can be formulated as
\begin{align}\label{L}\tag{$L$}
\min_{x\in\R^n}[g(x)-h(x)],
\end{align}
with functions $g,h:\;\real^n\to\real_+$ defined as
\begin{align}
g(x):=\mathds{I}_{\mathcal{P}}(x)+\sum_{m=1}^{\Mo}\wo_m\gamma_{\Bo_m}(x-\ao^{m}),&&
h(x):=\sum_{m=1}^{\Mu}\wu_m\gamma_{\Bu_m}(x-\au^{m}),
\end{align}
where $\mathds{I}_{\mathcal{P}}$ denotes the indicator function that attains the value 0, if $x\in\mathcal{P}$, and $+\infty$ otherwise. The parameters $\ao^1,\ldots,\ao^{\Mo}\in\real^n$, $(\Mo\geq 1)$ denote the attracting points with weights $\wo_1,\ldots,\wo_{\Mo}>0$, and  $\au^1,\ldots,\au^{\Mu}\in\real^n$, $(\Mu\geq 1)$ denote the repulsive points with weights $\wu_1,\ldots,\wu_{\Mu}>0$. The {polyhedral} unit balls $\Bo_1,\ldots,\Bo_{\Mo}$ and $\Bu_1,\ldots,\Bu_{\Mu}$ induce gauge distances that we associate with $\ao^1,\ldots,\ao^{\Mo}$ and $\au^1,\ldots,\au^{\Mu}$, respectively. %\bigskip 

For a compact polyhedral set $B\subseteq\real^n$, which contains the origin in its interior, the polyhedral gauge distance $\gamma_B:\;\real^n\to\real$ from $a\in\real^n$ to $x\in\real^n$  is defined as $\gamma_B(x-a):=\min\set{\lambda\geq	0\middle|\,x-a\in\lambda B}$, see
\cite{DM85,Rockafellar72}.\bigskip

The epigraph of $g$ is given by
\begin{align*}
	\epi g
	&=\set{(x,r)\in\mathcal{P}\times\real|\; r\geq \sum_{i=1}^{\Mo}\wo_i\gamma_{\Bo_i}(x-\ao^{i})}\\
	&=\set{(x,r)\in\mathcal{P}\times\real|\; r\geq \sum_{i=1}^{\Mo}\wo_i\min\set{\lambda_i\geq0|\;x-\ao^i\in\lambda_i\Bo_i }}\\
	&=\set{(x,r)\in\mathcal{P}\times\real|\;\exists \lambda\in\real^{\Mo}_+:\; r\geq \sum_{i=1}^{\Mo}\lambda_i\wo_i,\;x-\ao^i\in\lambda_i\Bo_i }\\
	&=\set{(x,r)\in\real^{n+1}|\;\exists \lambda\in\real^{\Mo}:\; r\geq \sum_{i=1}^{\Mo}\lambda_i\wo_i,\;\hat{\overline{B}}_i(x-\ao^i)\leq\lambda_ie;\; Px\geq p;\;\lambda \geq 0 },
\end{align*}
where $\hat{\overline{B}}_1,\ldots,\hat{\overline{B}}_{\Mo}$ are the matrices defining the polyhedral unit balls $\Bo_1,\ldots,\Bo_{\Mo}$, respectively, i.e., $x\in\Bo_i$ if and only if $\hat{\overline{B}}_ix\leq e$ for $ (i=1,\ldots,\Mo)$.\\\smallskip

A representation $\overline{A}=(\overline{B},\overline{b},\overline{C},\overline{c})$ of $g$ is given by
\begin{alignat*}{4}
\overline{B}=\begin{pmatrix}
-\hat{\overline{B}}_1\\
\vdots\\
-\hat{\overline{B}}_{\Mo}\\\hline
0\\
\vdots\\
0\\\hline
0\\\hline
P
\end{pmatrix},\quad
&
\overline{b}=\begin{pmatrix}
0\\\vdots\\0\\\hline
0\\\vdots\\0\\\hline
1\\\hline
0
\end{pmatrix},\quad
&
\overline{C}=\begin{pmatrix}
e&&\\
&\ddots&\\
&&e\\\hline
1&&\\&\ddots&\\&&1\\\hline
-\wo_1&\ldots&-\wo_{\Mo}\\\hline
0&\ldots&0
\end{pmatrix},\quad
&
\overline{c}=\begin{pmatrix}
-\hat{\overline{B}}_1\ao^1\\
\vdots\\
-\hat{\overline{B}}_{\Mo}\;	\ao^{\Mo}\\\hline
	0\\\vdots\\0\\\hline
	0\\\hline
	p
\end{pmatrix}.
\end{alignat*}
A representation of $h$ is given likewise.

The  epigraphs $\epi g^*$ and $\epi h^*$ and corresponding representations can easily be obtained by applying Example \ref{ex0d}  and Proposition \ref{rep_conj}. Moreover,    Assumptions {\ref{ass_primal}~}\eqref{a1}--\eqref{a4} in Section \ref{sec_g} (case $g$ polyhedral) and {\ref{ass_dual}~}\eqref{b1}--\eqref{b6}  in Section \ref{sec_h} (case $h$ polyhedral) are satisfied and hence, Algorithms \ref{primal_alg} and \ref{dual_alg} can be applied in order to solve the location problem \eqref{L}.

We note further that the conjugates of $g$ and $h$ can be written as
\begin{align}\label{g_star}
	g^*(y)&=\min\cbgg{\sum_{i=1}^{\Mo}\left[{\ao^i}^Ty^i+\mathds{I}_{\wo_i\Bo_i^*}(y^i)\right] +\sigma_{\mathcal{P}}(y^0)\bigg\vert\; \sum_{i=0}^{\Mo} y^{i}=y}\\
		h^*(y)&=\min\cbgg{\sum_{i=1}^{\Mu}\left[{\au^i}^Ty^i+\mathds{I}_{\wu_i\Bu_i^*}(y^i)\right] \bigg\vert\; \sum_{i=1}^{\Mu} y^{i}=y}
\end{align}
\cite{WMLT16}, where 
	$B^*:=\set{y\in\real^n\middle|\,\forall x\in B:\;x^Ty\leq 1}$
defines the dual unit ball of $B$. If $B$ is polyhedral then so is $B^*$.

We next solve several instances of \eqref{L} in the plane.
In dependence of the number of attracting and repulsive points, Tables \ref{tab:1} and \ref{tab:2} show the computational results of the primal and dual algorithm, respectively. In most cases, the primal algorithm performs better. 

\begin{remark}
In Step (2) of Algorithm \ref{dual_alg}, $g^*(y^i)$ is calculated for all vertices $(y^i, s^i)$, $i=1,\ldots,k$, of $\epi h^*$, obtained in Step (1). In general, we cannot directly calculate these objective values, see for example the representation \eqref{g_star}, which constitutes a linear program. In the general setting, the convex program \eqref{p3} needs to be solved for each vertex. On the other hand, in Step (2) of Algorithm \ref{primal_alg}, we need to calculate $h(x^i)$ for all vertices $(x^i, r^i)$, $i=1,\ldots,k$, of $\epi g$, obtained in Step (1). If (in case both $g$ and $h$ are polyhedral) only a representation of $h$ (i.e. a P-representation $\epi h$) is given, a general procedure to get the values $h(x^i)$ is to solve the linear program $\min r$ s.t. $(x^i,r) \in \epi h$. However, in the present example (and often in practice) one can directly calculate these objective values, i.e. there is no additional linear program to be solved. This is an advantage of the primal algorithm for our example from locational analysis.	
\end{remark}

Nevertheless, we also observe for the primal algorithm that an increasing number of attracting points influences the running time much more than an increasing number of repulsion points. It is vices versa for the dual algorithm, which leads to settings where the dual algorithm performs even better, such as $\Mo=100$, $\Mu=20$. Thus, both algorithms have their justification.
	
\begin{table}[htp]
\begin{center}
\begin{tabular}{r@{\hskip 1cm}rrrrrrr}
\toprule 
& \multicolumn{7}{c}{$\Mu$} \\ \cmidrule(r){2-8} 
$\Mo$ & 20& 40& 60& 80& 100& 1000& 5000\\ 
\midrule 
20& 0.37& 0.51& 0.53& 0.47& 0.86& 5.16& 52.55\\ 
40& 3.36& 2.89& 5.77& 4.88& 6.77& 22.07& 176.67\\ 
60& 17.92& 23.29& 34.54& 25.34& 37.55& 91.24& 462.10\\ 
80& 74.16& 96.38& 117.54& 154.28& 152.81& 291.03& 899.93\\ 
100& 292.85& 337.41& 289.85& 342.32& 387.56& 945.23& 1797.00\\ 
\bottomrule
\end{tabular}
\caption{Computational results (running time in seconds) for Example \ref{ex:loc_an} using the primal algorithm (case $g$ polyhedral).}
\label{tab:1}
\end{center}
\end{table}
\begin{table}[tph]
\begin{center}
\begin{tabular}{r@{\hskip 1cm}rrrrr}
\toprule 
& \multicolumn{5}{c}{$\Mu$} \\ \cmidrule(r){2-6} 
$\Mo$ & 20& 40& 60& 80& 100\\ 
\midrule 
20& 3.84& 16.64& 63.50& 198.76& 573.40\\ 
40& 10.13& 40.87& 111.28& 358.97& 915.98\\ 
60& 18.17& 103.39& 216.55& 593.61& 1436.48\\ 
80& 44.35& 192.98& 465.88& 943.77& 1778.65\\ 
100& 58.85& 338.85& 688.54& 1316.08& 2793.30\\ 
1000& 91.52& 406.74& 1005.94& 2065.33& 4503.63\\ 
5000& 2077.18& 8082.09& -& -& -\\ 
\bottomrule
\end{tabular}
\caption{Computational results (running time in seconds) for Example \ref{ex:loc_an} using the dual algorithm (case $h$ polyhedral).}
\label{tab:2}
\end{center}
\end{table}	
\end{example}

\begin{example}\label{ex2}
Ferrer, Bagirov and Baliakov \cite{Ferrer2014} introduce the DC extended cutting angle method (DCECAM) and compare it with the DC prismatic algorithm (DCPA) described in \cite{Ferrer2005}. We compare Algorithms \ref{primal_alg} and \ref{dual_alg} with these two methods by means of Example 10 in  \cite{Ferrer2014}: Consider a DC programming problem \eqref{p} with
	$$ g(x) = |x_1-1|+200 \sum_{i=2}^n  \max \cb{0,|x_{i-1}|-x_i}$$
	and
	$$ h(x) = 100 \sum_{i=2}^n  \of{|x_{i-1}|-x_i}.$$
	In Table \ref{tab:3} we present numerical results for the primal and dual algorithm developed in this article and compare them with the results of DCECAM and DCPA obtained in \cite{Ferrer2014}. We see that both of our algorithms perform better than DCECAM and DCPA, where it should be noted, that the latter algorithms are designed for a more general problem class (non-polyhedral DC problems). In particular, the dual method is preferable for this example. Note that for $n=10$ we need to solve a multiple linear program with $n+2=12$ objectives. 
	\begin{table}[htp] 
	\begin{center}
	\begin{tabular}{rrrrr}
	\toprule 
	 n        & DCECAM \cite{Ferrer2014} & DCPA \cite{Ferrer2014} & primal method & dual method \\ 
	\midrule
	 2        & 0.21            & 0.22          & 0.04 & 0.03 \\ 
	 3        & 3.57            & 4.63          & 0.09 & 0.04 \\ 
	 4        & 2.47            & 0.78          & 0.29 & 0.05 \\ 
	 5        & 345.12          & 502.29        & 2.05 & 0.09 \\ 
	 6        & -               & -             & 65.76 & 0.32 \\ 
	 7        & -               & -             & 5107.84 & 1.17 \\ 
	 8        & -               & -             & -       & 47.36 \\ 
	 9        & -               & -             & -       & 247.21 \\ 
	 10       & -               & -             & -       & 3028.10\\ 
	\bottomrule
	\end{tabular}
	\caption{Computational results for Example 10 in \cite{Ferrer2014}. The running times (in seconds) for DCECAM and DCPA are taken from \cite{Ferrer2014} and are obtained on a computer with Intel\textregistered\ Core\texttrademark\ i5-3470S CPU with 2.9 GHz.}
	\label{tab:3}
	\end{center}
	\end{table}	
\end{example}

\begin{example}\label{ex:03} Let $Q \in \R^{n \times n}$ be a positive semi-definite symmetric matrix.
	The problem to maximize the convex quadratic function $h:\R^n \to \R$, $h(x) = x^T Q x$ over a polyhedral convex set $S \subseteq \R^n$ can be reformulated as a DC-program by letting $g=\mathds{I}_S$, i.e. the indicator function of $S$, and minimizing $f=g-h$. For example, let $S:= \cb{x \in \R^n \st -e \leq x  \leq e}$ be an $n$-dimensional hypercube. As $h$ is convex, $Q$ has only non-negative eigenvalues. We assume that $Q$ has at most $m \leq n$ positive eigenvalues. Equivalently, $Q$ can be expressed as $Q=P^T P$ for some $P \in \R^{m \times n}$. For example, let the entries of $P$ be defined by
	$$ P_{ij} = 
	\lfloor m \cdot \sin((j-1)\cdot m+i)\rfloor,$$
where  $\lfloor x\rfloor := \max\cb{z \in \mathbb{Z}\mid z \leq x}$. 
Setting $y:=Px$  and $T = \cb{y \in \R^m \mid \exists x \in S: y = Px}$, we obtain the equivalent problem to minimize $\bar h:\R^m \to \R \cup \cb{+\infty}, \bar h(y) = y^T y$ over the polyhedron $T\subseteq \R^m$. 
For $n=10$ and $m=4$ we have 
$$ P=\begin{pmatrix}
% 3    -4     1     1    -4     3    -1    -3     3    -3
%3    -2    -3     3    -4    -1     3    -4     2     1
%0     2    -4     2     0    -4     3    -2    -2     3
%-4     3    -3    -2     3    -4     1     2    -4     2
     \phantom{-}3&    -4&     \phantom{-}1&     \phantom{-}1&    -4&     \phantom{-}3&    -1&    -3&     \phantom{-}3&    -3\\
     \phantom{-}3&    -2&    -3&     \phantom{-}3&    -4&    -1&     \phantom{-}3&    -4&     \phantom{-}2&     \phantom{-}1\\
   \phantom{-}0&     \phantom{-}2&    -4&     \phantom{-}2&     \phantom{-}0&    -4&     \phantom{-}3&    -2&    -2&     \phantom{-}3\\
    -4&     \phantom{-}3&    -3&    -2&     \phantom{-}3&    -4&     \phantom{-}1&    \phantom{-}2&    -4&     \phantom{-}2\\
\end{pmatrix}$$
and the optimal solution obtained by Algorithm \ref{primal_alg} applied to the reformulated problem is  $$x = ( 1,    -1,     1,     1,    -1,     1,    -1,    -1,     1,    -1)^T.$$
In Table \ref{tab:ex03}, further numerical results for this example can be found.
\end{example}
	\begin{table}[htp] 
	\begin{center}
	\begin{tabular}{r@{\hskip 1cm}rrrrrr}
	\toprule 
	& \multicolumn{5}{c}{$m$} \\ \cmidrule(r){2-6} 
	$n$ & 2 & 3 & 4 & 5 & 6\\ 
	\midrule 
	10    & 0.02 & 0.04 & 0.6 & 8.7 & 76.0\\ 
	50    & 0.02 & 1.57 & 22.5 &  -  & -  \\ 
	200   & 0.04 & 2.39 &  62440.0    & -   & -  \\ 
	1000  & 0.15 & 3.10 & -    & -   & -  \\ 
	5000  & 1.50 & 13.60 & -    & -   & -  \\ 	
	\bottomrule
	\end{tabular}
	\caption{Running times in seconds for Example \ref{ex:03}.}
	\label{tab:ex03}
	\end{center}
	\end{table}
		
\begin{example} \label{ex:04} Let $g:\R^n\to \R$, $g(x) = x^T Q x$ for some positive definite symmetric matrix $Q$. Then the conjugate is $g^*(y) = \frac{1}{4} y^T Q^{-1} y$ and the optimal solution of \eqref{p3} is $x=\frac{1}{2}Q^{-1} y$. For example, let $Q=P^T P$ where $P \in \R^{n \times n}$ defined by $P_{ij} = 1$ for $i \geq j$ and $P_{ij}=0$ otherwise, then $Q$ is positive definite. Further, we set
	$$ h(x) = \sum_{i=2}^n  \of{|x_{i-1}|-x_i}.$$
Some numerical results of minimizing $f = g - h$ using Algorithm \ref{dual_alg} are shown in Table \ref{tab:ex04}. 	
\end{example}

	\begin{table}[htp] 
	\begin{center}
	\begin{tabular}{r@{\hskip 1cm}rrrrrrrrr}
	\toprule 	
	$n$ & 2 & 3 & 4 & 5 & 6 & 7 & 8 & 9 & 10\\ 
	\midrule 
	optimal value   & -1.25 & -2.75 & -3.75 & -4.75 & -5.75& -6.75 & -7.75 & -8.75 & -9.75 \\
	running time (sec)  & 0.02 & 0.02 & 0.02 & 0.02 & 0.05 & 0.21 & 1.85 & 21.5 & 281 \\
	\bottomrule
	\end{tabular}
	\caption{Numerical results for Example \ref{ex:04}.}
	\label{tab:ex04}
	\end{center}
	\end{table}

\section{Conclusion and final remarks}

It is demonstrated that MOLP solvers provide a very easy way to solve DC programs where at least one of the two functions $g,h$ in the objective $f=g-h$ is polyhedral. In case of both $g$ and $h$ being polyhedral, the investigations result in two different algorithms, called primal and dual algorithm. The running time of the primal algorithm depends mainly on the effort to compute the vertices of $\epi g$ by the MOLP solver. Likewise the dual algorithm depends on $\epi h^*$. The proposed methods work well, if the dimension of $\epi g$ in Algorithm \ref{primal_alg} or $\epi h^*$ in Algorithm \ref{dual_alg} is not too high (up to $11$ in the examples above).
\medskip

{In this study we consider functions decomposed as $f = g-h$ where $g$ (resp. $h$) is polyhedral. The examination of such a decomposition is motivated by applications such as the problem to locate semi-obnoxious facilities, see Example \ref{ex:loc_an}. Nevertheless, one may wish to define the class of all functions that have such a structure and formulate an algorithm to obtain it. 
	It is known that twice continuously differentiable functions are DC \cite{Har59,HU85}. A DC function $f=g-h$ with  $g$ being polyhedral (resp. $h$ being polyhedral) is piecewise concave (resp. convex). Here ``piecewise'' means that the domain of $f$ can be decomposed into finitely many convex polyhedra $P_i$ such that they have disjoint interior and $f$ is concave (resp. convex) on each $P_i$. Moreover, $f$ is locally Lipschitz on the interior of its domain as so is every concave (or convex) function. An interesting question is whether or not the opposite implication holds: 
	Can each piecewise concave (resp. convex) function $f$ which is locally Lipschitz on the interior of its domain be decomposed as $f = g-h$ where $g$ (resp. $h$) is polyhedral? To answer this question and also to provide an appropriate algorithm is left as an open problem.
}

{\bigskip \noindent {\bf Acknowledgements:} We thank both anonymous reviewers for their useful comments which inspired the Remarks \ref{rev_rem1} and \ref{rev_rem2} and the discussion about the class of DC functions whose decomposition contains a polyhedral component in the last paragraph of the conclusions.}

\bibliographystyle{abbrv}
\bibliography{database,Convex_Analysis,DC,Location,Repulsion}

\end{document}

%% file: def.tex
\newtheorem{theorem}{Theorem}
\newtheorem{corollary}[theorem]{Corollary}
\newtheorem{remark}[theorem]{Remark}
\newtheorem{algorithm+}[theorem]{Algorithm}
\newtheorem{assumptions}[theorem]{Assumptions}

\newtheorem{proposition}[theorem]{Proposition}
\newtheorem{ex}[theorem]{Example}
\newenvironment{example}{\begin{ex} \em }{\em \end{ex}}
\newtheorem{definition}[theorem]{Definition}

\newcommand{\cb}    [1]{\ensuremath{\left  \{      #1  \right \}       }}

\newcommand{\cbgg}  [1]{\ensuremath{\biggl \{      #1  \biggr \}       }}

\newcommand{\of}    [1]{\ensuremath{\left (        #1  \right )        }}

\newcommand{\st} {\ensuremath{|\;}}
\newcommand{\smz}{\!\setminus\!\{0\}}

\DeclareMathOperator*{\epi} {epi}

\DeclareMathOperator*{\cl} {cl}
\DeclareMathOperator*{\bd} {bd}
\DeclareMathOperator*{\Int} {int}
\DeclareMathOperator*{\conv} {conv}

\DeclareMathOperator*{\gr} {gr}
\DeclareMathOperator*{\vertex}{vert}
\DeclareMathOperator*{\dom} {dom}

\DeclareMathOperator*{\argmin}{argmin}

\newcommand{\PP}{\mathcal{P}}

\newcommand{\R}{\mathbb{R}}
\newcommand{\N}{\mathbb{N}}

%% file: Format.tex
\usepackage{dsfont}
\usepackage[ngerman,english]{babel}
\newcommand{\real}{\mathbb{R}}

\newcommand{\ao}{\overline{a}}
\newcommand{\au}{\underline{a}}
\newcommand{\Bo}{\overline{B}}
\newcommand{\Bu}{\underline{B}}

\newcommand{\Mo}{\overline{M}}
\newcommand{\Mu}{\underline{M}}

\newcommand{\wo}{\overline{w}}
\newcommand{\wu}{\underline{w}}

\newcommand{\set}[1]{\left\{ #1\right\}}